\documentclass[12pt]{article}
\usepackage{amssymb,amsmath,amsthm}
\usepackage{graphicx}
\usepackage{subcaption}
\usepackage{fullpage}
\usepackage{scrextend}
\usepackage{siunitx}
\usepackage{enumerate}
\usepackage{import}
\usepackage{standalone}
\usepackage{tikz,calc}
\usepackage{tikz-cd}
\usepackage{subcaption}
\usetikzlibrary{calc}
\usepackage{epsfig,graphicx}
\usepackage{listings}

 \newtheorem{theorem}{Theorem}%[section]

\newtheorem{lemma}[theorem]{Lemma}
\newtheorem{cor}[theorem]{Corollary}

\newtheorem{defn}[theorem]{Definition}

\title{On rigidity of unit-bar frameworks}
\author{J\'ozsef Solymosi \qquad Ethan White\\
Department of Mathematics\\
The University of British Columbia \\
Vancouver, BC\\
 Canada V6T 1Z2}
\begin{document}

\maketitle

\begin{abstract}

We show the existence of infinitesimally rigid bipartite unit-bar frameworks in $\mathbb{R}^d$. We also construct unit-bar frameworks with girth up to 12 that are infinitesimally rigid in the plane. This answers problems proposed by Maehara. 

\end{abstract} 

\section{Introduction} 

The \textit{unit distance problem} was posed by Paul Erd\H{o}s in 1946: how many pairs of $n$ points in the plane can be unit distance apart? \cite{ed} Erd\H{o}s gave a construction that proved there are at least $n^{1+o(1)}$ such pairs, and he conjectured that this is the true order of magnitude. This is one of the central open problems in discrete geometry. Understanding point configurations with many unit distances is an important problem. 

A \textit{framework} in $\mathbb{R}^d$ is a graph with vertices that are distinct points in $\mathbb{R}^d$, and edges that are line segments between vertices. We refer to the vertices of a framework as \textit{joints} and edges as \textit{bars}. A framework is \textit{flexible} if there is a continuous motion of its joints, keeping bar lengths constant, while changing the distance between two non-adjacent joints. If a framework is not flexible, it is \textit{rigid}. For example, in the plane a square can be deformed into a family of rhombi, and so it is flexible. On the other hand, the shape of a triangle is uniquely determined by the lengths of its three sides, and so it is rigid.

An \textit{infinitesimal motion of} $\mathbb{R}^d$ is a vector field $f \colon \mathbb{R}^d \to \mathbb{R}^d$ such that for all pairs of points $x,y \in \mathbb{R}^d$:
\begin{equation}\label{infdef}
\left( f(x) - f(y) \right) \cdot (x-y) = 0 .
\end{equation}

Let $F$ be a framework in $\mathbb{R}^d$ with joints $X$. An \textit{infinitesimal motion of the framework} $F$ is a vector field $g \colon X \to \mathbb{R}^d$ that satisfies (\ref{infdef}) for all bars $xy$ in $F$. If every infinitesimal motion $g$ of the framework $F$ is of the form $f|_X$ for some infinitesimal motion $f$ of $\mathbb{R}^d$, then we say $F$ is \textit{infinitesimally rigid}, otherwise $F$ is \textit{infinitesimally flexible}. 

A framework possessing a continuous motion also admits a smooth motion, see \cite{ar}. The initial velocity of the joints in a framework undergoing a smooth continuous motion is an infinitesimal motion. Hence flexible frameworks are infinitesimally flexible and infinitesimally rigid frameworks are rigid. 

A \textit{unit-bar framework} has bars of only one length. Constructing rigid unit-bar frameworks can be done by attaching equilateral triangles, but determining rigid triangle-free unit-bar frameworks is harder. Maehara constructed a rigid bipartite unit-bar framework in \cite{m2} with 353 joints and 676 bars. His construction is rigid, but not infinitesimally rigid. In \cite{mc} Maehara and Chinen find an infinitesimally rigid triangle-free unit-bar framework with 22 joints and 41 bars. Their framework contains pentagons. In \cite{m1} and \cite{mc} the authors propose the following problems:
\begin{enumerate}[i.]
\item Find an infinitesimally rigid bipartite unit-bar framework in the plane
\item Find a general method to construct a triangle-free, infinitesimally rigid unit-bar framework in $\mathbb{R}^d$. 
\end{enumerate}
In this paper we solve these problems. In Section 2 we show a method for constructing infinitesimally rigid bipartite unit-bar frameworks in $\mathbb{R}^d$. In Section 3 we construct infinitesimally rigid bipartite unit-bar frameworks in the plane with girth up to 12. Our calculations in Section 3 rely on computers. For sake of completeness we provide the computer code in Appendix 2. The python files of our programs have been uploaded alongside this paper. 

%%%

\section{Infinitesimally rigid unit-bar frameworks in $\mathbb{R}^d$}

In our constructions we use variants of the knight's graph. The knight's graph has a vertex for each square on a chessboard and edges that represent legal moves the knight. 

\begin{defn}\label{knightdef} The $m \times n$ knight's framework in $\mathbb{R}^2$ has a joint at all integer coordinates $(x,y)$ where $0 \leq x \leq m-1$ and $0 \leq y \leq n-1$. Two joints $(x_1,y_1)$, $(x_2,y_2)$ have a bar between them if $|x_1-x_2| = 1$ and $|y_1-y_2| = 2$, or if $|x_1 - x_2| = 2$ and $|y_1-y_2| = 1$. We will denote the $m \times m$ knight's framework by $N_m$. 

\end{defn}

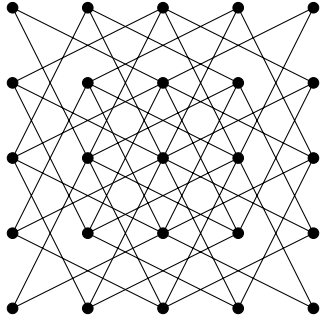
\begin{figure}[h!]
    \centering
       \begin{tikzpicture}[scale=1]
\draw
 (0,0) -- (2,1)
(0,0) -- (1,2)
(0,1) -- (2,2)
(0,1) -- (1,3)
(0,1) -- (2,0)
(0,2) -- (1,4)
(0,2) -- (2,3)
(0,2) -- (1,0)
(0,2) -- (2,1)
(0,3) -- (1,1)
(0,3) -- (2,2)
(0,3) -- (2,4)
(0,4) -- (1,2)
(0,4) -- (2,3);
\draw
(1,0) -- (3,1)
(1,0) -- (2,2)
(1,1) -- (3,2)
(1,1) -- (2,3)
(1,1) -- (3,0)
(1,2) -- (2,4)
(1,2) -- (3,3)
(1,2) -- (2,0)
(1,2) -- (3,1)
(1,3) -- (2,1)
(1,3) -- (3,2)
(1,3) -- (3,4)
(1,4) -- (2,2)
(1,4) -- (3,3)
(2,0) -- (4,1)
(2,0) -- (3,2)
(2,1) -- (4,2)
(2,1) -- (3,3)
(2,1) -- (4,0)
(2,2) -- (3,4)
(2,2) -- (4,3)
(2,2) -- (3,0)
(2,2) -- (4,1)
(2,3) -- (3,1)
(2,3) -- (4,2)
(2,3) -- (4,4)
(2,4) -- (3,2)
(2,4) -- (4,3)
(3,0) -- (4,2)
(3,1) -- (4,3)
(3,2) -- (4,4)
(3,2) -- (4,0)
(3,3) -- (4,1)
(3,4) -- (4,2);
\filldraw
(0,0) circle (2pt)
(0,1) circle (2pt)
(0,2) circle(2pt)
(0,3) circle(2pt)
(0,4) circle(2pt)
(1,0) circle(2pt)
(1,1) circle(2pt)
(1,2) circle(2pt)
(1,3) circle(2pt)
(1,4) circle(2pt)
(2,0) circle(2pt)
(2,1) circle(2pt)
(2,2) circle(2pt)
(2,3) circle(2pt)
(2,4) circle(2pt)
(3,0) circle(2pt)
(3,1) circle(2pt)
(3,2) circle(2pt)
(3,3) circle(2pt)
(3,4) circle(2pt)
(4,0) circle(2pt)
(4,1) circle(2pt)
(4,2) circle(2pt)
(4,3) circle(2pt)
(4,4) circle(2pt);

\end{tikzpicture}
        \caption{$5 \times 5$ knight framework}
        
        \end{figure}

The knight's framework is a unit-bar framework. Two joints $(x_1,y_1)$, $(x_2,y_2)$ are adjacent only if $x_1+y_1$ and $x_2+y_2$ have different parity, and so the framework is bipartite. An infinitesimally rigid framework in the plane on $v$ joints must have at least $2v-3$ bars \cite{gss}. The $m \times n$ knight's framework has $2(m-1)(n-2)+2(m-2)(n-1)$ bars. It is easy to check that the smallest $m\times n$ knight's framework with enough edges to be rigid is the $5 \times 5$ framework. 

The infinitesimal motions of $\mathbb{R}^d$ arise from the initial velocities of smooth rigid motions, i.e. rotations and translations. As a result, the space of infinitesimal motions of $\mathbb{R}^d$ has dimension $\binom{d+1}{2}$. A framework $F$ is infinitesimally rigid if and only the space of infinitesimal motions of $F$ has dimension $\binom{d+1}{2}$. 

\begin{theorem}\label{5x5} The $5 \times 5$ knight's framework is infinitesimally rigid. 
\end{theorem}

The reader can skip the proof and refer to the program of Appendix 2, where the rigidity of $N_5$ is verified using the rigidity matrix. The rank of the rigidity matrix can also be computed without computer aid; however, it is a system of $50$ variables. The following lemma reduces the number of variables and facilitates a shorter by-hand proof of Theorem~\ref{5x5}.

\begin{lemma} (Rhombus Lemma) \label{rhombus} Let $p_1p_2p_3p_4$ be a framework of a non-degenerate rhombus in the plane. If $v_1,v_2,v_3,v_4$ are the velocity vectors associated with any infinitesimal motion of the rhombus, then $v_1+v_3 = v_2+v_4$. 

\end{lemma}

\begin{proof} Put $x = p_2-p_1 = p_3 - p_4$ and $y = p_3-p_2 = p_4-p_1$. We have:

\begin{align*}
(v_2-v_1) \cdot x & = 0, \\
(v_3-v_2) \cdot y & = 0, \\
(v_4-v_3) \cdot x & = 0, \\
(v_1-v_4) \cdot y & = 0. \\
\end{align*}
The first and third equation give $(v_1+v_3) \cdot x = (v_2+v_4) \cdot x$, while the second and fourth give $(v_1+v_3) \cdot y = (v_2+v_4) \cdot y$. Since $x$ and $y$ are linearly independent, we have the desired result. 
\end{proof}

If $f \colon \mathbb{R}^d \to \mathbb{R}^d$ is a function, let $f_k(x)$ denote the value in the $k^{th}$ coordinate of $f(x)$.

\begin{proof}[Proof of Theorem~\ref{5x5}] Let the joints of $N_5$ from left to right, top to bottom be $p_1,p_2,\ldots,p_{25}$. Consider all infinitesimal motions $f$ of $N_5$ such that
\begin{equation}\label{2freeze}
f(p_{13}) = f_1(p_2)=0.
\end{equation}
This specifies three degrees of freedom of $f$, so the dimension of the space of infinitesimal motions of $N_5$ that satisfy (\ref{2freeze}) is at most three less than the dimension of the space of all infinitesimal motions of $N_5$. Since the space of infinitesimal motions of the plane has dimension three, if all infinitesimal motions $f$ of $N_5$ that satisfy (\ref{2freeze}) are identically zero then $N_5$ is infinitesimally rigid. Let $f$ be an infinitesimal motion of $N_5$ satisfying (\ref{2freeze}) and put $f(p_i) = v_i$ for all $i$. Since $p_2p_{13}$ is a bar we have that $v_2 = 0$. Using Lemma~\ref{rhombus}, we are able to determine all velocities $v_i$ homogenously in terms of the velocities $v_4,v_6,v_{10},v_{20}$, and $v_{22}$. The first equation in every line below follows from an application of Lemma~\ref{rhombus} to a rhombus in $N_5$, a second equation in any line is a substitution of a previous equation. We have:
\begin{align}
v_3 &= v_6+v_{10} \nonumber \\
v_{11} &=  v_{22} \nonumber \\
v_{15} &= v_4 + v_{24} \nonumber \\
v_{23} &= v_{16}+v_{20} \nonumber \\
v_7 &= v_4+v_{16}\nonumber \\
v_9 &= v_{20}\nonumber \\
v_{17} &= v_6+v_{24}\nonumber \\
v_{19} &=v_{10}+v_{22}\nonumber \\
v_8 &= v_{11} + v_{19} - v_{22} =  v_{10} + v_{22}\label{8.1} \\
v_8 &= v_{15} + v_{17} - v_{24} = v_4+v_6+v_{24} \label{8.2} \\
v_{12} &= v_9+v_{23} - v_{20} =  v_{16} + v_{20} \label{12.1} \\
v_{12} &=v_3 + v_{19} - v_{10} = v_6 + v_{10} + v_{22} \label{12.2} \\
v_{14} &= v_3 + v_{17} - v_6 = v_6 + v_{10} + v_{24}\label{14.1} \\
v_{14} &= v_7+v_{23} - v_{16} = v_4 + v_{16} + v_{20}\label{14.2} \\
v_{18} &= v_7+v_{15} - v_4 = v_4 + v_{16} + v_{24}\label{18.1} \\
v_{18} &= v_9+v_{11}- v_2 =  v_{20} + v_{22}\label{18.2} \\
v_1 &= v_8+v_{12} - v_{19} =  v_{16} + v_{20}\nonumber \\
v_5 &= v_8 + v_{14}  - v_{17} =  2v_{10} + v_{22}\nonumber \\
v_{21} &= v_{12} + v_{18} - v_9 = v_4 + 2v_{16} + v_{24}\nonumber \\
v_{25} & = v_{14} + v_{18} - v_7 = v_6 + v_{10} + 2v_{24}\nonumber \\
v_{11} + v_{15} &= v_8+ v_{18} \Rightarrow  v_{16} = - v_{10}\nonumber \\
v_3+v_{23} & = v_{12} + v_{14} \Rightarrow v_{24} = 0.\nonumber 
\end{align}
Equating equations (3),(4) and (9),(10) gives 
\begin{align}
v_{10} + v_{22} & = v_4 + v_6 + v_{24} \\
v_4 + v_{16} + v_{24} &=  v_{20} + v_{22} \nonumber .
\end{align}
Adding the above equations gives $v_6 + v_{20} =v_{10} + v_{16} =  0$. Equating equations (5),(6) and (7),(8) gives
\begin{align}
 v_{16} + v_{20} & = v_6 + v_{10} + v_{22} \\
v_6 + v_{10} + v_{24} &= v_4 + v_{16} + v_{20} \nonumber .
\end{align}
Adding the above equations gives $v_4 + v_{22} =  v_{24} = 0$. Substituting into (11) and (12) we obtain:
\begin{align*}
v_{10}- v_{4} & = v_4 + v_{6} \\
 - v_{10} - v_{6} &= v_6 + v_{10} - v_{4}  .
\end{align*}
The above system gives $v_4 = \frac{4}{5}v_{10}$ and $v_6 = -\frac{3}{5}v_{10}$. Now we see that all velocities are scalar multiples of $v_{10}$. Since $p_{10}p_{13}$ is a bar we have that $v_{10} \cdot (p_{10} - p_{13}) = 0$. Since $p_3p_{10}$ is a bar we have that $(v_3-v_{10}) \cdot (p_3-p_{10}) = v_6 \cdot (p_3-p_{10}) = -\frac{3}{5} v_{10} \cdot (p_3-p_{10}) = 0$. The directions of the bars $p_{10}p_{13}$ and $p_3p_{10}$ are linearly independent, and so $v_{10} = 0$. It follows that all velocities are zero and $N_5$ is infinitesimally rigid. 
\end{proof} 

The framework obtained by deleting the corner joints and one degree three joint from the $5 \times 5$ knight's framework is also infinitesimally rigid. This framework has 20 joints and 37 edges. The rigidity of this framework can be verified using a similar approach to the above, or by calculating the rank of its rigidity matrix. Every joint in the $5 \times 6$ knight's framework that is not in the $5 \times 5$ framework has two bars in linearly independent directions connecting it to the $5 \times 5$ framework, and so the $5 \times 6$ framework is infinitesimally rigid. Inductively we see that the $m \times n$ knight's framework is infinitesimally rigid for all $m,n \geq 5$. The knight's framework can be extended to higher dimensions.

%%%

\begin{defn} An $n$-lattice framework in $\mathbb{R}^d$ has joints of the form $(x_1,\ldots,x_d)$, where $x_i \in \{0,1,\ldots,n-1\}$. Let $F$ be an $n$-lattice framework in $\mathbb{R}^d$. Define $F_{i,c}$ to be the subframework of $F$ induced by all joints in $F$ of the form $(x_1,\ldots,x_{i-1},c,x_{i+1},\ldots,x_{d})$. The framework $F_{i,c}$ can be embedded in $\mathbb{R}^{d-1}$ by contracting the $i^{th}$ coordinate of all joints. The resulting framework is an $n$-lattice framework in $\mathbb{R}^{d-1}$, call it $F_{i,c}'$. 
\end{defn} 

The frameworks $F_{i,c}$ are slices of the framework $F$. In Figure 2, the slice $F_{1,1}$ is infinitesimally flexible, but $F_{1,1}'$ is infinitesimally rigid. 

\begin{figure}[h!]
    \centering
    \begin{subfigure}[t]{0.5\textwidth}
        \centering
     \begin{tikzpicture}[scale=1]
     
     \filldraw
     (0,0,0) circle (2pt) 
     (3,0,0) circle (2pt) node[align=center,above right] {$(1,1,0)$}
     (0,3,0) circle (2pt)
     (3,3,0) circle (2pt)
     (0,0,3) circle (2pt)
     (3,0,3) circle (2pt)
     (0,3,3) circle (2pt)
     (3,3,3) circle (2pt) node[align=center,below left] {$(1,0,1)$};
     \draw
     (0,0,0)--(3,0,0)
     (0,0,0)--(0,3,0)
     (0,0,0)--(0,0,3)
     (3,3,3)--(0,3,3)
     (3,3,3)--(3,0,3)
     (3,3,3)--(3,3,0)
     (3,3,0)--(3,0,0)
     (3,3,0)--(0,3,0)
     (3,0,3)--(3,0,0)
     (3,0,3)--(0,0,3)
     (0,3,3)--(0,3,0)
     (0,3,3)--(0,0,3)
     (3,0,0)--(3,3,3)
     (3,3,0)--(3,0,3);

     \end{tikzpicture}
        \caption{A 2-lattice framework $F$ in $\mathbb{R}^3$}
    \end{subfigure}%
    ~ 
    \begin{subfigure}[t]{0.5\textwidth}
        \centering
     \begin{tikzpicture}[scale=1]
     \filldraw
     (0,0) circle (2pt) --
     (3,0) circle (2pt) --
     (3,3) circle (2pt) --
     (0,3) circle (2pt) --
     (0,0);
     \draw
     (0,0) -- (3,3) 
     (3,0) --(0,3);

    \end{tikzpicture}
        \caption{The slice $F_{1,1}$}
    \end{subfigure}
    \caption{}
\end{figure}

\begin{lemma}\label{span} Let $y,x_1,x_2,\ldots,x_n$ be joints of a framework $F$ such that $yx_i$ is a bar for all $i$. Let $f$ be an infinitesimal motion of $F$ such that $f(x_i) = 0$ for all $i$. If $z$ is in the span of $\{y-x_1,\ldots,y-x_n\}$, then $f(y) \cdot z = 0$. 
\end{lemma}
\begin{proof} Let $z = a_1(y-x_1) + \cdots + a_n(y-x_n)$ with $a_i \in \mathbb {R}$. Since $yx_i$ is an edge, $\left( f(y) - f(x_i) \right) \cdot (y-x_i) = f(y) \cdot (y-x_i) =0$ for all $i$. Hence
\[ f(y) \cdot z = a_1 f(y) \cdot (y-x_1) + \cdots + a_n f(y) \cdot (y-x_n) = 0.\]
\end{proof}

When the dimension is unambiguous, we will use the notation $e_k$ to represent the standard basis vector consisting of a $1$ in the $k^{th}$ entry and zeroes elsewhere. The vector $e_k$ will represent both the direction, and the joint with the corresponding coordinates. The context will make the use clear. 

\begin{theorem}\label{completeslice} Let $F$ be an $n$-lattice framework in $\mathbb{R}^d$, $d \geq 3$, and $n \geq 2$. If for all $1 \leq i \leq d$ and $0 \leq c \leq n-1$, the framework $F_{i,c}$ has bars between all pairs of joints, then $F$ is infinitesimally rigid. 
\end{theorem}

\begin{proof} Consider all infinitesimal motions $f$ of $F$ such that 
\begin{equation}\label{pinned}
f(0) = 0, \quad \text{and} \quad f_i(e_k) = 0 \quad \text{for} \quad 1 \leq k \leq d-1 \quad \text{and} \quad k+1 \leq i \leq d .
\end{equation}
The restrictions of (\ref{pinned}) specify $d+(d-1) + \ldots + 1 = \binom{d+1}{2}$ degrees of freedom of $f$. Hence the space of infinitesimal motions of $F$ that satisfy (\ref{pinned}) is at most $\binom{d+1}{2}$ less than the dimension of the space of all infinitesimal motions of $F$. It follows that if the only infinitesimal motions of $F$ that satisfy (\ref{pinned}) are identically zero, then $F$ is infinitesimally rigid. \\

Let $f$ be an infinitesimal motion of $F$ satisfying (\ref{pinned}). Note that $e_10$ is a bar of $F$ and $e_1$ is in the span of $\{e_1-0\}$. Since $f(0) = 0$, by Lemma~\ref{span} we see that $ f(e_1) \cdot e_1= 0$ and so $f(e_1) = 0$. Notice $e_i0$ is a bar for all $1 \leq i \leq d$. For all $j \neq i$, since $d\geq 3$, we have that $e_ie_j$ is also a bar. A simple induction and Lemma~\ref{span} gives the result $f(e_i) = 0$ for all $1 \leq i \leq d$. For any joint $x \in F_{i,0}$ we have that $x0$ and $xe_j$ are bars for all $j \neq i$. Lemma~\ref{span} gives $f_j(x) = 0$ for all $j \neq i$. Hence if a joint $x$ has a zero in two or more coordinates, $f(x) = 0$. Let $y = (y_1,\ldots,y_d)$ be a joint of $F$ such that $y_i \neq 0$ for all $i$. Let $y^{(i)}$ be the joint with $y_i$ in the $i^{th}$ coordinate and zeros in all other coordinates. Notice that $yy^{(i)}$ is a bar for all $1 \leq i \leq d$. Furthermore, since $d \geq 3$, $y^{(i)}$ has a zero in at least two coordinates, and so $f(y^{(i)}) = 0$. It is easy to check that the span of $\{ y - y^{(i)} \}_{1 \leq i \leq d}$ is all of $\mathbb{R}^d$, and so by Lemma~\ref{span}, $f(y) = 0$. Finally, let $x = (x_1,\ldots,x_d)$ be a joint of $F$ such that $x_i = 0$ and $x_j \neq 0$ for $j \neq i$. Let $z$ be the joint $(x_1,\ldots,x_{i-1},1,x_{i+1},\ldots,x_d)$. The existence of $z$ follows from $n \geq 2$. Since $xz$ is a bar: 
\[ \left(f(x) - f(z) \right) \cdot (x-z) = \left(f(x) - f(z) \right) \cdot e_i = 0 .\]
Since all coordinates of $z$ are nonzero, $f(z) = 0$, and so $f_i(x) = 0$. It follows that $f \equiv 0$, and $F$ is an infinitesimally rigid framework.

\end{proof}

\begin{cor}\label{infrigslice} Let $F$ be an $n$-lattice framework in $\mathbb{R}^d$, $d \geq 3$ and $n \geq 2$. If for all $1 \leq i \leq d$, and $0 \leq c \leq n-1$, the framework $F_{i,c}'$ is infinitesimally rigid, then $F$ is infinitesimally rigid. 

\begin{proof} Let $1 \leq i \leq d$ and $0 \leq c \leq n_i-1$ be arbitrary. Any infinitesimal motion $f$ of $F$ induces an infinitesimal motion $f_{i,c}$ of $F_{i,c}$ in the following way. For any joint $x \in F_{i,c}'$ let $\hat{x}$ denote the corresponding joint in $F_{i,c}$, and put 
\[ f_{i,c}(x) = \left[ f_1(\hat{x}) \ldots  f_{i-1}(\hat{x}) \ f_{i+1} (\hat{x}) \ldots f_d(\hat{x}) \right]^t.\]
It is clear that this defines $f_{i,c}$ as a vector field in $\mathbb{R}^{d-1}$. Furthermore, for any bar $xy$ of $F_{i,c}'$, since $f$ is an infinitesimal motion:
\begin{equation}\label{slice}
\left( f_{i,c} (x) - f_{i,c}(y) \right) \cdot (x-y) = \left( f(\hat{x}) - f(\hat{y}) \right) \cdot (\hat{x}-\hat{y}) = 0. 
\end{equation}
It follows that $f_{i,c}$ is an infinitesimal motion. Notice that the first equality in (\ref{slice}) holds for all $x,y \in F_{i,c}'$, and not just bars. Since $F_{i,c}'$ is infinitesimally rigid we see that both equalities in (\ref{slice}) holds for all $x,y \in F_{i,c}$. Hence all infinitesimal motions of $F$ are infinitesimal motions of the framework described in Theorem~\ref{completeslice}, and so $F$ is infinitesimally rigid.

\end{proof}

\end{cor}

\begin{defn} The $n\times \cdots \times n$ knight's framework in $\mathbb{R}^d$ is the $n$-lattice framework with bars between two joints $x$ and $y$ if the coordinates of $x$ and $y$ are equal except in two places where they differ by $1$ and $2$. 
\end{defn}

All bars in the knight's framework have length $\sqrt{5}$. The parity of the sum of the coordinates of two adjacent joints is different, the same as in the two dimensional case. Hence the knight's framework in $\mathbb{R}^d$ is bipartite, and in particular, triangle free. A consequence of Theorem~\ref{5x5} and Corollary~\ref{infrigslice} is the following. 

\begin{theorem} The $ 5 \times \cdots \times 5$ knight's framework in $\mathbb{R}^d$, for $d \geq 2$, is an infinitesimally rigid bipartite unit-bar framework. 
\end{theorem}

Using a computer and the rigidity matrix we noticed that the $4 \times 4 \times 4$ knight's framework is infinitesimally rigid. The computer code of this program can be found in Appendix 1. It follows that the $4 \times \cdots \times 4$ knight's framework in $\mathbb{R}^d$ for $d \geq 3$ is also infinitesimally rigid by Corollary~\ref{infrigslice}. 

%%%%%%%%%%%%%%%%%%%%%%%%

\section{Unit-bar bipartite frameworks with higher girth}

Erd\H{o}s' construction of many unit distances motivated our approach to finding infinitesimally rigid unit-bar frameworks with larger girth. We consider subframeworks of an $n \times n$ lattice of joints with bars of length $\sqrt{m}$, where $m$ can be written as the sum of two squares in several ways. For odd $m$, two numbers summing to $m$ have different parity. Hence the sum of the coordinates of adjacent joints is different, and the framework is bipartite. One can show that for even $m$ a framework constructed in this way is also bipartite, see for example \cite{b}. The following algorithm gives an outline of how we construct our frameworks. \\

\textbf{Algorithm:}\\
Input: The size $n$ of the square lattice, an integer $m$ that can be written as the sum of two squares in several ways, and the desired girth $2g$.  \\
Output: A bipartite unit-bar framework with girth at least $2g$. 
\begin{enumerate}[(1)]
\item Determine all ordered pairs of integers $(a,b)$ where $a^2+b^2 = m$ and either $b >0$, or $b = 0$ and $a>0$. These are the bar directions, call the set $D$. 
\item Add the joints to the framework, they are at the integer coordinates $(x,y)$ with $0 \leq x,y \leq n-1$. 
\item Find a permutation $\sigma$ of the joints. For each joint $x$ make a list $D(x) = D$ of all possible directions of bars.
\item In the order described by $\sigma$ visit each joint $x$ and do the following. 
\begin{enumerate}[i.] 
\item Randomly select an untried bar direction $d$ from $D(x)$, let $y = x+d$. 
\item If $y$ is a joint in the framework then determine all joints within distance $g-1$ of $x$ and distance $g-2$ of $y$, call these sets $N_x$ and $N_y$. 
\item If $N_x$ and $N_y$ are disjoint then add the bar $xy$ to the framework.  
\item Remove $d$ from $D(x)$. 
\end{enumerate}
\item Repeat (4) until $D(x)$ is empty for all joints $x$, this will take $|D|$ loops. 
\item Remove joints with degree less than three. Output the framework. 
\end{enumerate}

\textbf{Implementation:} We used Python to construct frameworks according to the above algorithm. We tested the infinitesimal rigidity of the outputted framework using the rigidity matrix. The infinitesimal motions of a framework $F$ in $\mathbb{R}^d$ can be described by the nullspace of the rigidity matrix of $F$. The rigidity matrix of a framework with $v$ joints has $vd$ columns. If the nullspace of the rigidity matrix has dimension $\binom{d+1}{2}$ then the framework is infinitesimally rigid. Equivalently, if the rank of the rigidity matrix is $vd - \binom{d+1}{2}$, then the framework is infinitesimally rigid. For more on the rigidity matrix see \cite{gss}. We used built-in functions of Python and Matlab to determine the rank of the rigidity matrix. Not all frameworks constructed according to our algorithm are rigid. For each girth, we experimented with different $m$ and $n$, and used many random trials. The following table describes the smallest infinitesimally rigid framework of each girth we found.

\begin{center}
 \begin{tabular}{||c | c | c | c | c | c ||} 
 \hline
 Girth & Size $n$ & $m$ & \# of Joints & \# of Edges & \# of Trials \\ [0.5ex] 
 \hline\hline
 4 & 5 & $5 = 1^2 + 2^2$ & 21 & 40 & 1 \\ 
 \hline
 6 & 9 & $5 = 1^2 + 2^2$ & 54 & 105 & 16000000\\
 \hline
 8 & 23 & $\begin{aligned} 65 & = 1^2 + 8^2 \\ &= 4^2+7^2 \end{aligned}$ & 436 & 869 & 600000 \\
 \hline
 10 & 53 & $\begin{aligned} 1105 & = 4^2+ 33^2 \\ &= 9^2+33^2 \\ &= 12^2+31^2 \\ &= 23^2+24^2 \end{aligned}$ & 2467 & 4931 & 5000 \\
 \hline
 12 & 147 & $\begin{aligned} 5525 & = 7^2 + 74^2 \\ &= 14^2+73^2 \\ &= 22^2+71^2 \\ &= 25^2+70^2\\ &= 41^2+62^2 \\ &= 50^2+55^2\end{aligned}$ & 18924 &  37845 & 10 \\ [1ex] 
 \hline
\end{tabular}
\end{center}

The Python script we used to construct frameworks and test rigidity is in Appendix 2. For the frameworks with girth 4,6,8 and 10 we used Matlab's rank function to double-check the rank calculations of Python. The Matlab function `svds' computed the smallest singular value of the rigidity matrix of our girth 12 framework to be 0.0005. This value was reproduced upon decreasing the convergence tolerance and increasing the number of iterations of the svd algorithm. This calculation indicates that all singular values of the rigidity matrix are nonzero and the framework is infinitesimally rigid. \\

Below we draw the frameworks in the above table with girth 4,6, and 8. For these frameworks we also record their adjacency matrices below by representing ones with black squares and zeros with white squares. For the frameworks with girth 10 and 12 we record their adjacency matrices using darker shading to represent higher density of edges. 

%%Girth4

\begin{figure}[h]
    \centering
    \begin{subfigure}{0.3\textwidth}
        \centering
         \includegraphics[scale=1]{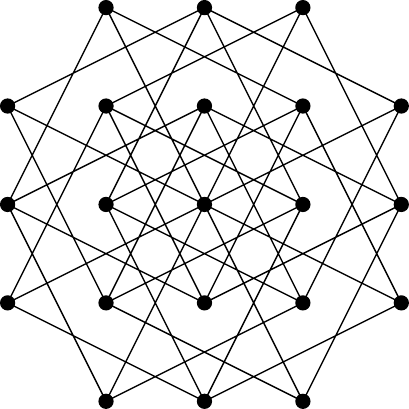}
     %  \includestandalone[width = 0.85\textwidth]{girth4frame}

        \caption{Framework }
    \end{subfigure}%
    ~ 
    \begin{subfigure}{0.3\textwidth}
        \centering
         \includegraphics[scale=0.2,angle=90]{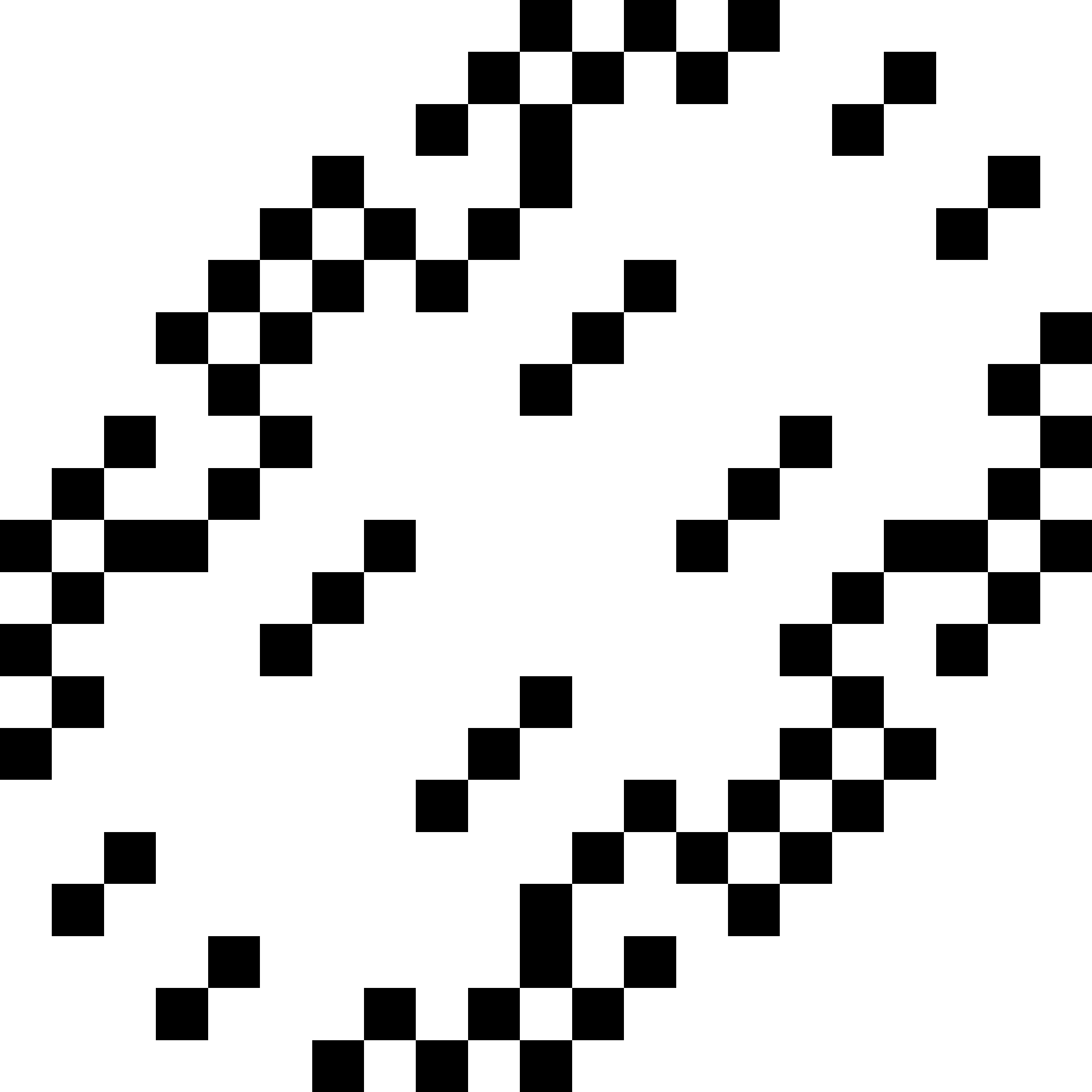}
  % \includestandalone[width = 0.85\textwidth,angle = 90]{girth4adj}
        \caption{Adjacency matrix}
    \end{subfigure}
    \caption{Girth 4}
\end{figure}

%%Girth6

\begin{figure}[h!]
    \centering
    \begin{subfigure}{0.35\textwidth}
        \centering
    \includegraphics[scale=0.64]{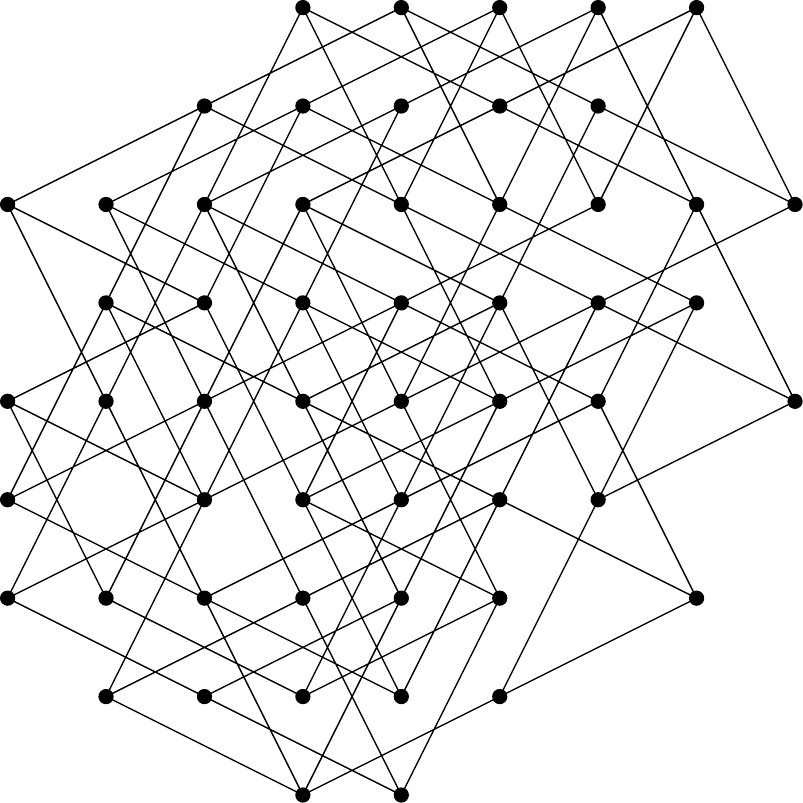}

        \caption{Framework }
    \end{subfigure}%
    ~ 
    \begin{subfigure}{0.35\textwidth}
        \centering
     
      \includegraphics[scale=0.097,angle=90]{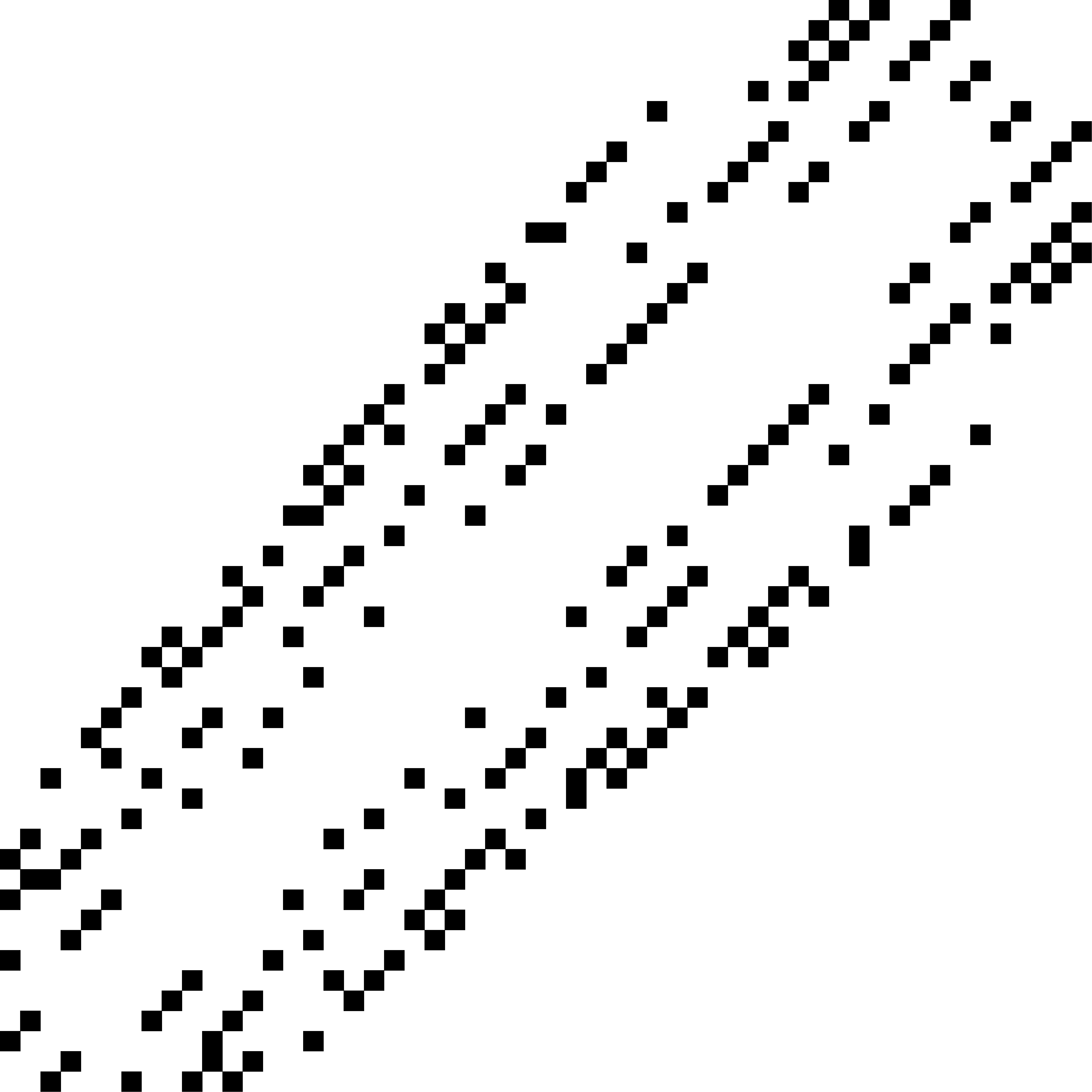}

        \caption{Adjacency matrix}
    \end{subfigure}
    \caption{Girth 6}
\end{figure}

%Girth8

\begin{figure}[h!]
    \centering
    \begin{subfigure}{0.4\textwidth}
        \centering

    \includegraphics[scale=0.24]{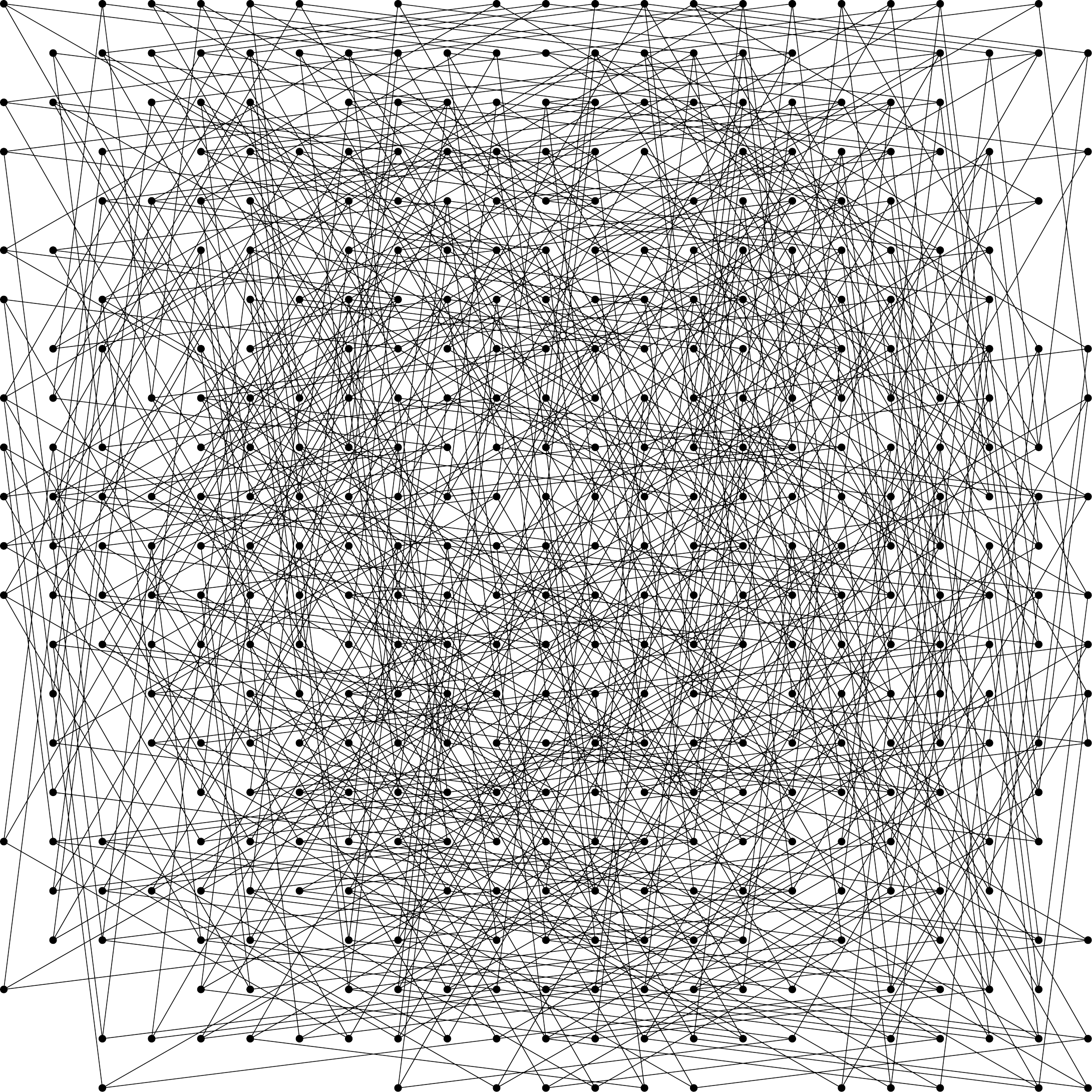}

        \caption{Framework }
    \end{subfigure}%
    ~ 
    \begin{subfigure}{0.4\textwidth}
        \centering

    \includegraphics[scale=0.0122,angle=90]{girth8adj.pdf}

        \caption{Adjacency matrix}
    \end{subfigure}
    \caption{Girth 8}
\end{figure}

%%Girth10 and 12?

\begin{figure}[h!]
    \centering
    \begin{subfigure}{0.4\textwidth}
        \centering
    \includegraphics[scale=0.00975,angle=90]{girth10adj.pdf}

        \caption{Girth 10 }
    \end{subfigure}%
    ~ 
    \begin{subfigure}{0.4\textwidth}
        \centering

    \includegraphics[scale=0.00875,angle=90]{girth12adj.pdf}

        \caption{Girth 12}
    \end{subfigure}
    \caption{}
\end{figure}

%%%%%%%%%%%%%%%%%%%%%%%%%%%%%%%%%%%%%%%%%%

\section{Problems}

We are limited by computational power in finding infinitesimally rigid frameworks of higher girth. We expect they exist. \\

\textbf{Problem 1.} Construct an infinitesimally rigid unit-bar framework with arbitrarily large girth. \\

The knight's graph is one instance of an $(a,b)$-leaper graph. This graph has vertices for each square of an $m\times n$ chessboard, and edges for squares with coordinates that differ by $a$ and by $b$. Knuth showed that if $a+b$ and $a-b$ are relatively prime, then for sufficiently large chessboards, the $(a,b)$-leaper graph is connected \cite{k}. The $(a,b)$-leaper framework can be defined analogously to Definition~\ref{knightdef}. We have verified the following for $1 \leq a,b \leq 25$. \\

\textbf{Problem 2.} Prove that the $(a,b)$-leaper framework on an $m \times n$ chessboard is infinitesimally rigid if and only if $a+b$ is relatively prime to $a-b$ and $m,n \geq 2(a+b)-1$. \\

We expect that with more random trials smaller rigid frameworks can be found. \\

\textbf{Problem 3.} Determine the fewest number of joints in a infinitesimally rigid unit-bar framework for each girth $g \geq 4$.

\section{Acknowledgements} The research of the second author was supported in part by an NSERC CGS M. The research of the first author was supported in part by an NSERC Discovery grant and OTKA NK grant. The work of the first author was also supported by the European Research Council (ERC) under the European Union's Horizon 2020 research and innovation programme (grant agreement No. 741420, 617747, 648017).

%%%%%%%%%%%%%%%%%%%%%%%%%%%%%
%\medskip

%%%%%%%%%%%%%%%%%%%%%%%%%%%%%%%%%%%%%%%%%%
%%%%%%%%%%%%%%%%%%%%%%%%%%%%%%%%%%%%%%%%%%
%%%%%%%%%%%%%%%%%%%%%%%%%%%%%%%%%%%%%%%%%%

\section{Appendix 1}

This python program is used to verify the rigidity of the knight's framework in $\mathbb{R}^3$, in particular the $4 \times 4 \times 4$ knight's framework. 

\begin{lstlisting}[basicstyle = \small]
#This program determines if the knight graph
#in three dimensions is rigid.
#L,M,N are the dimensions of the lattice
#a,b are the lengths of the knight's leaps

from numpy.linalg import matrix_rank

#Specify parameters

L=4
M=4
N=4
a=1
b=2

#Construct matrix of zeros with dimensions of rigidity matrix

numvert = L*M*N
numedges = (N*(2*(L-a)*(M-b)+2*(L-b)*(M-a))+L*(2*(M-a)*(N-b)
            +2*(M-b)*(N-a))+M*(2*(N-a)*(L-b)+2*(N-b)*(L-a)))

matrix = []

for r in range(numedges):
    matrix.append([])
    for c in range(3*N**3):
        matrix[r].append(0)

#Create a list of the possible edge directions
edges = []
for twoplace in range(0,3):
    for oneplace in range(0,3):
        if twoplace != oneplace:
            edge = []
            for i in range(3):
                edge.append(0)
            edge[twoplace] = b
            edge[oneplace] = a
            edges.append(edge)

            edge = []
            for i in range(3):
                edge.append(0)
            edge[twoplace] = -b
            edge[oneplace] = a
            edges.append(edge)

#This function determines if a coordinate is
#inside the lattice

def inrange(coor):
    good = 0

    if coor[0] not in list(range(L)):
        good +=1
    if coor[1] not in list(range(M)):
        good +=1
    if coor[2] not in list(range(N)):
        good +=1

    return good

#Place all edges into the matrix,
#entry is position of current row to be added
entry = 0

for y in range(0,N):
    for x in range(0,M):
        for w in range(0,L):
            for e in edges:
                w2 = w + e[0]
                x2 = x + e[1]
                y2 = y + e[2]
                #If edge is in lattice, then add to matrix
                if inrange([w2,x2,y2])== 0:
                
                    cw = (M*L*y+L*x+w)*3
                    cw2 = (M*L*y2+L*x2+w2)*3

                    matrix[entry][cw] = (-1)*e[0]
                    matrix[entry][cw+1] = (-1)*e[1]
                    matrix[entry][cw+2] = (-1)*e[2]
     
                    matrix[entry][cw2] = e[0]
                    matrix[entry][cw2+1] = e[1]
                    matrix[entry][cw2+2] = e[2]      
                    entry+=1

print('Required rank for rigidity:')
print(3*L*M*N-6)
print('Rank of rigidity matrix:')
print(matrix_rank(matrix))

\end{lstlisting}

%%%%%%%%%%%%%%%%%%%%%%%%%%%%%%%%%%%%%%%%

\section{Appendix 2}
The following python program is our implementation of the algorithm outlined in Section 3. We provide comments in the script that reference the steps described in the algorithm.

\begin{lstlisting}[basicstyle = \small]
from collections import deque
from numpy.random import permutation
import random
from numpy.linalg import matrix_rank

#Input n the size, m the square of bar length, g the girth

n = int(input('Dimension of Grid (n): '))
m = int(input('Square length of bars (m): '))
g = int(input('Girth of Framework (g): '))
num_trials = int(input('Number of Trials: '))
use_python_rank ='y'
#for large frameworks, we use function 'svds' in matlab
if num_trials ==1:
    use_python_rank = input('Use python rank? (y/n): ')
    
#determine g-1 and g-2
g1 = g//2 -1
g2 = g1-1

#Step 1: Determine all bar directions
D = [[],[]]
for a in range(0,m+1):
    for b in range(a,m+1):
        if a**2+b**2 == m:
            if a==b:
                D[0].append(a)
                D[1].append(a)
                D[0].append(-a)
                D[1].append(a)
            elif a==0:
                D[0].append(b)
                D[1].append(0)
                D[0].append(0)
                D[1].append(b)
            else:
                D[0].append(b)
                D[1].append(a)
                D[0].append(a)
                D[1].append(b)
                D[0].append(-a)
                D[1].append(b)
                D[0].append(-b)
                D[1].append(a)
num_directions = len(D[0])

def makeframework():
    F = {}
    checked = {}
    #Step 2: Add joints to framework as keys to a dictionary
    #Values of the dictionary are neighbours (empty)
    for v in range(0,n**2):
        F[v] = []
        checked[v] = []
        for d in range(num_directions):
            checked[v].append(d)
    #Step 3 get a permutation of the joints
    order = permutation(n**2)
    rounds = 0
    #The below loop is Step 5
    while rounds < num_directions:
        #Step 4
        for v1 in order:
            #Get coordinates of joint
            y1 = v1//n
            x1 = v1 - n*y1
            #Step 4.i
            t = random.choice(checked[v1])
            x2 = x1 + D[0][t]
            y2 = y1 + D[1][t]
            #Step 4.ii
            if x2 <= n-1 and x2 >= 0 and y2 <= n-1 and y2 >= 0:
                v2 = x2 + y2*n
                N_v1 = neighbourhood(v1,g1,F)
                N_v2 = neighbourhood(v2,g2,F)
                #Step 4.iii
                if set(N_v1).isdisjoint(N_v2):
                    F[v1].append(v2)
                    F[v2].append(v1)
            #Step 4.iv
            checked[v1].remove(t)
        rounds += 1

    #Step 6
    [minjoint,mindegree] = getmindegree(F)
    while mindegree <3:
        for v in F:
            if minjoint in F[v]:
                F[v].remove(minjoint)
        del F[minjoint]
        [minjoint,mindegree] = getmindegree(F)

    return F

#This is the function used in Step 6
def getmindegree(frame):
    minjoint = -1
    mindegree = 3
    for v in frame:
        deg = len(frame[v])
        if deg< mindegree:
            mindegree = deg
            minjoint = v

    return [minjoint,mindegree]
        
#This is the function used in Step 4.ii 
def neighbourhood(initial,distance,frame):
    queue = deque([initial])
    depth = {initial:0}
    visited = [initial]
    explored = []

    while queue:
        current = queue.popleft()
        explored.append(current)
        
        if depth[current]<distance:
            for neighbour in frame[current]:
                if neighbour not in visited:
                    queue.append(neighbour)
                    visited.append(neighbour)
                    depth[neighbour]=depth[current]+1

    return explored

#matrix creates a matrix out of the framework
#python can compute the rank of it
def matrix(frame,num_vert,num_edges):
    R = []
    for r in range(num_edges):
        R.append([])
        for c in range(2*n**2):
            R[r].append(0)
    currentrow = 0
    for v1 in frame:
        for v2 in frame[v1]:
            if v1<v2:
                
                v1y = v1//n
                v1x = v1-v1y*n
                v2y = v2//n
                v2x = v2-v2y*n
                
                R[currentrow][v1*2] = v1x-v2x
                R[currentrow][v1*2+1] = v1y-v2y
                R[currentrow][v2*2] = v2x-v1x
                R[currentrow][v2*2+1] = v2y-v1y
                currentrow+=1
    return R

#'sparsematrix' writes rigidity matrix data to a file
#we used this output in matlab
def sparsematrix(frame):
    places = {}
    count = 1
    for v in frame:
        places[v] = count
        count +=1

    sparse_file = open('sparsematrix.txt','w')
    for v1 in frame:
        for v2 in frame[v1]:
            if v1<v2:
                v1y = v1//n
                v1x = v1 - v1y*n
                v2y = v2//n
                v2x = v2 - v2y*n
                c1 = 2*places[v1]-1
                c2 = 2*places[v2]-1
                sparse_file.write(str(c1)+','+str(c2)+','+str(v1x-v2x)+','
                +str(v1y-v2y)+','+str(v2x-v1x)+','+str(v2y-v1y)+'\n')
    sparse_file.close()

#'saveframe' 
def saveframe(frame):
    frame_file = open('framefile.txt','w')
    
    for v1 in frame:
        for v2 in frame[v1]:
            if v1<v2:
                frame_file.write(str(v1)+','+str(v2)+'\n')
    frame_file.close()

num_found = 0
def dotrials():
    smallest_size = n**2
    smallest_edges = 0
    smallest_frame = {}
    

    for trial in range(num_trials):
        aframe = makeframework()
        num_vert = len(aframe)
        num_edges = 0
        for v in aframe:
            num_edges += len(aframe[v])
        num_edges = num_edges//2
        if num_edges >= 2*num_vert-3:
            amatrix = matrix(aframe,num_vert,num_edges)
            rank = matrix_rank(amatrix)
            if 2*num_vert -3 == rank:
                num_found+=1
                if num_vert<smallest_size:
                    smallest_size = num_vert
                    smallest_edges = num_edges
                    smallest_frame = aframe
                
    return [smallest_frame,smallest_size,smallest_edges]
    
def main():
    if use_python_rank == 'y':
        [frame,size,edges] = dotrials()
        if num_found == 0:
            print('No rigid frameworks found')
        else:
            print('The smallest rigid framework found has', size , 'joints')
            print('and', edges, 'edges')

    else:
        frame = makeframework()
        size = len(frame)
        edges = 0
        for v in frame:
            edges += len(frame[v])
        edges = edges//2
        print('The framework found has',size,'joints')
        print('and',edges,'edges')
    if num_found >0:       
        print('Saving sparse matrix to file...')
        sparsematrix(frame)
        print('Completed')
        print('Saving framework to file...')
        saveframe(frame)
        print('Completed')

main()
\end{lstlisting}

\end{document}